\documentclass[11pt]{amsart}
\usepackage{amssymb}
\usepackage{latexsym}
\usepackage{graphics}
\usepackage{mathptmx}
\usepackage{amsmath}
\usepackage{amsxtra}
\usepackage{amstext}
\usepackage{amssymb}
\usepackage{latexsym}
\usepackage{dsfont}
\usepackage{verbatim}
\usepackage{color}
\DeclareMathOperator{\capa}{cap}

\newcommand{\defeq}{\mathrel{\mathop:}=}

\newcommand{\E}{{\mathcal E}}
\newcommand{\I}{{\mathbf I}}

\newcommand{\K}{{\mathcal K}}

\newcommand{\G}{{\mathbf G}}
\newcommand{\R}{{\mathbb R}}

\newcommand{\M}{\mathcal{M}^{+}(\R^n)}

\begin{document}
\title[Sublinear equations and Schur's test ] {Sublinear equations and Schur's test\\ 
for integral operators}

\author{ Igor E. Verbitsky}
\address{Department of Mathematics, University of Missouri, Columbia, MO  65211, USA}
\email{verbitskyi@missouri.edu}

\begin{abstract}
	We study weighted norm inequalities of $(p,r)$-type, 
		\[ \Vert \mathbf{G} (f \, d \sigma) \Vert_{L^r(\Omega, d\sigma)} \le C \Vert f \Vert_{L^p(\Omega, \sigma)}, \quad \text{ for all } f \in L^p(\sigma), \]
			for $0 < r < p$ and $p>1$, where $\mathbf{G}(f d \sigma)(x)=\int_\Omega G(x, y) f(y) d \sigma(y)$  is an integral operator associated with a nonnegative kernel $G(x,y)$ on $\Omega\times \Omega$, and $\sigma$  is a locally finite 
			positive measure in $\Omega$. 
	
We show that this  embedding holds if and only if 
\[
\int_\Omega (\mathbf{G} \sigma)^{\frac{pr}{p-r}} d \sigma<+\infty,
\]
provided $G$ is a quasi-symmetric kernel which satisfies the weak maximum principle. 

In the case $p=\frac{r}{q}$, where $0<q<1$, we prove that this condition characterizes 
the existence of a non-trivial solution (or supersolution) $u \in L^r(\Omega, \sigma)$, for 
$r>q$, to the 
 the sublinear integral equation
		\[ u - \mathbf{G}(u^q \, d \sigma) = 0, \quad u \ge 0. \] 
		
		We also give some counterexamples  
		in the end-point case $p=1$, which corresponds to solutions 
		$u \in L^q (\Omega, \sigma)$ of this integral equation, studied recently in 
		\cite{QV1}, \cite{QV2}. These problems appear in the investigation of weak solutions to  the  sublinear equation involving the (fractional) Laplacian,  
		\[ (-\Delta)^{\alpha} u -  \sigma \, u^q = 0, \quad u \ge 0, \]
	for $0<q<1$ and  $0 < \alpha < \frac{n}{2}$ in domains $\Omega \subseteq \R^n$ with a positive Green function. 
\end{abstract}

\dedicatory{In memory of Professor Victor Havin}

\maketitle

\numberwithin{equation}{section}
\newtheorem{theorem}{Theorem}[section]
\newtheorem{lemma}[theorem]{Lemma}
\newtheorem{remark}[theorem]{Remark}
\newtheorem{cor}[theorem]{Corollary}
\newtheorem{prop}[theorem]{Proposition}
\newtheorem{defn}[theorem]{Definition}
\allowdisplaybreaks

\section{Introduction}
	\noindent
	Let $\Omega$ be a locally compact, Hausdorff space.  
	For a positive, lower semicontinuous kernel 
	$G\colon \Omega\times \Omega\to (0, +\infty]$, we denote by 
	$$\G (f \, d \sigma)(x) =\int_{\Omega} G(x, y) \, f(y) \, d \sigma(y), \quad x \in \Omega, $$ 
	 the corresponding integral operator, where  $\sigma \in \mathcal{M}^+(\Omega)$, the class of   
	 positive locally finite Radon measures in $\Omega$.  
	 
		We study the $(p, r)$-weighted norm inequalities 
	\begin{equation}\label{strong-type}
		\Vert \mathbf{G} (f d\sigma) \Vert_{L^r(\Omega, \sigma)} \le C \, \Vert f \Vert_{L^p(\Omega, \sigma)}, \quad \forall f \in L^p(\Omega, \sigma), 
	\end{equation}
	in the case $0<r<p$ and $p \ge 1$,  
where $C$ is a positive constant which does not depend on $f$.

	The main goal of this paper is to find explicit characterizations of  \eqref{strong-type} in terms of  $\mathbf{G} \sigma $ 
	under certain assumptions on $G$. We also  study connection  of  inequality \eqref{strong-type} with 
	$p=\frac{r}{q}$, where $0<q<1$, to 
	 the existence of a positive function $u \in L^r(\Omega, \sigma)$ such that 
	\begin{equation}
			\label{super_sol} 
			u \ge  \G(u^q \sigma) \quad d \sigma-\text{a.e.}  \, \,  \text{ in $\Omega$}, 
				\end{equation}
				in the case $r>q$. In other words, $u$ is a supersolution for the sublinear integral equation
		\begin{equation}
			\label{int_eqn} 
			u - \G(u^q \sigma) =0, \quad 0<u<+\infty \quad d \sigma-\text{a.e.}  \, \,  \text{ in $\Omega$}, 
				\end{equation}
				where $0<q<1$.

	In this paper, we assume that the kernel $G$ of the integral operator is quasi-symmetric, and  satisfies a weak maximum principle (WMP); see Sec. \ref{background}. 
	Such restrictions are satisfied by the Green kernel associated with many elliptic operators, including the fractional Laplacian 
	$(- \Delta)^{\alpha}$, as well as quasi-metric kernels, and 
	radially symmetric, decreasing convolution kernels 
	$G(x,y) = k(|x-y|)$ on $\mathbb{R}^n$ (see, e.g.,  \cite{AH}, \cite{An}, \cite{Maz}, \cite{QV1}, \cite{QV2} and the literature cited there).
	
	If $G$ is Green's kernel associated with the Laplacian in an open domain $\Omega \subseteq \mathbb{R}^n$, \eqref{int_eqn}  
	 is equivalent to  the sublinear elliptic boundary value problem    
		\begin{equation}
			\label{lap_eqn}
		 \begin{cases}
			- \Delta u - \sigma u^q=0, &u>0 \, \text{ in } \Omega, \\
			u = 0 & \text{ on } \partial \Omega,
		\end{cases} 
		\end{equation}
	where $0 < q < 1$. 
	
	We observe that solutions $u \in L^r(\Omega, \sigma)$ to \eqref{lap_eqn} in the case $r=1+q$ correspond to 
	finite energy solutions $u \in L^{1, 2}_{0} (\Omega)$ in the Dirichlet space, i.e., 
	$$
	\int_\Omega |\nabla u|^2 dx <+\infty,
	$$
	where $u$ has zero boundary values (see \cite{CV1}).

	The more difficult  end-point case $p=1$ of \eqref{strong-type},  
	along with  solutions  $u \in L^q(\Omega, \sigma)$ in the case $r=q$, was  studied recently 
	in \cite{QV1}, \cite{QV2}. After a certain modification, it leads to 
	solutions $u \in L^q_{{\rm loc}}(\Omega, \sigma)$, i.e., 
	all solutions to \eqref{int_eqn}, or \eqref{lap_eqn} understood in a weak sense (see \cite{MV}). For  Riesz kernels on $\Omega=\R^n$  such 
	$(1, q)$-weighted norm 
	inequalities,  along with  weak solutions to the 
	sublinear problem 
	\begin{equation}
			\label{frac_lap_eqn}
		 \begin{cases}
			(- \Delta)^{\alpha} u - \sigma u^q=0, \quad u>0 \, \text{ in } \R^n, \\
			\displaystyle{\liminf_{x \to \infty}} \,\,  u = 0,  \quad u \in L^q_{{\rm loc}}(\sigma), &
		\end{cases} 
		\end{equation}
	for $0<\alpha<\frac{n}{2}$, were
	treated earlier in \cite{CV1}, \cite{CV2}, \cite{CV3}. 
		
		Our main result is the following theorem. 
		
		\begin{theorem}\label{strong-thm}
		Let $\sigma \in \mathcal{M}^{+}(\Omega)$. 
		 Suppose $G$ is a positive, quasi-symmetric, lower semicontinuous kernel on $\Omega\times\Omega$ which satisfies the weak maximum principle.

		(i) If  $\, 1<p<+\infty$ and $0<r<p$, then the $(p, r)$-weighted norm inequality \eqref{strong-type} holds if and only if 
					\begin{equation}\label{r-p} 
\int_\Omega (\mathbf{G} \sigma)^{\frac{pr}{p-r}} d \sigma<+\infty. 
\end{equation}

		(ii) If  $\, 0<q<1$ and $q<r < \infty$,  then there exists a positive (super)solution $u\in L^r(\Omega, d \sigma)$ to \eqref{int_eqn} if and only if \eqref{strong-type} holds with $p=\frac{r}{q}$, or equivalently, 
		\begin{equation}\label{r-q} 
\int_\Omega (\mathbf{G} \sigma)^{\frac{r}{1-q}} d \sigma<+\infty. 
\end{equation} 
		\end{theorem}
	
	\begin{remark}\label{rm1} {\rm We observe that the ``if'' parts of statements (i) and (ii) of Theorem~\ref{strong-thm} 
	 fail if $p=1$, and $r=q$, respectively. 	The ``only if'' parts hold for all $0<r<p$ in statement (i), 
	 and $r>0$ in statement (ii).}
	    \end{remark} 
	    
	    \begin{remark}\label{rm2}
	    {\rm It is known that inequality \eqref{strong-type} with $p=\frac{r}{q}\ge 1$ in the case $0<q<1$ yields the existence of a positive supersolution $u\in L^r(\Omega, \sigma)$ for \eqref{super_sol}. This statement  follows from a  lemma due to Gagliardo \cite{G}, and does not require $G$ to be quasi-symmetric or to satisfy the WMP (see Sec. \ref{sec3} below). However, the converse statement does not hold without the WMP (see  \cite{QV2} in the case $r=q$).}
	    \end{remark}
	
	 	\begin{remark}\label{rm3}  {\rm Without the assumption that $G$ satisfies the WMP, 
	 the ``only if'' parts of statement (i) (with $p=\frac{r}{q}\ge 1$) and 
	statement   (ii) (with $r \ge q$) hold only for $0<r \le 1-q^2$ (see 
	 Lemma \ref{lemma2} below).} 
	 \end{remark}
	 
	 In particular, if there exists a positive (super)solution $u \in L^q(\Omega, \sigma)$, then 
	 \eqref{r-q} holds 
	 with $r=q$ for $0<q \le q_0$, where $q_0 = \frac{\sqrt{5}-1}{2} = 0.61\ldots$ 
	 is the conjugate golden ratio. However, \eqref{r-q}  with $r=q$ generally 
	 fails (even for symmetric 
	 kernels) in the case $q_0<q<1$; the cut-off $q=q_0$ here is sharp  \cite{QV2}. 
	 
	 In  Sec. \ref{sec2} below,   
	 we discuss 
	 related results,  and provide some counterexamples in the case $p=1$.

\section{Kernels and potential theory}\label{background}

	Let $G\colon  \Omega \times  \Omega \rightarrow (0, +\infty]$ be a positive  kernel. 
	We will assume that $\Omega$ is a locally compact space Hausdorff space, and $G$ is lower semicontinuous, so that we can 
	apply elements of the classical potential theory developed for such kernels (see \cite{Brelot},   \cite{F}). Most of our 
	results hold for \textit{non-negative} kernels $G(x, y) \ge 0$.  In that case, some statements concerning the existence of positive solutions (rather than supersolutions)  require the additional assumption that $G$ is  non-degenerate; see \cite{QV2}.

	By $\mathcal{M}^+(\Omega)$ we denote the class of all nonnegative, locally finite, Borel measures on $\Omega$. We  use the notation $\operatorname{supp}(\nu)$ for the support of $\nu \in \mathcal{M}^+(\Omega)$ and $\Vert \nu \Vert = \nu (\Omega)$ if $\nu$ is a finite measure.

	For $\nu \in \mathcal{M}^+(\Omega)$, the potential of $\nu$ is defined by
		\[ \mathbf{G}\nu (x) \defeq \int_{\Omega} G(x,y) d\nu(y), \quad \forall x \in \Omega, \]
		and  the potential with the adjoint kernel
		\[ \mathbf{G}^*\nu (y) \defeq \int_{\Omega} G(x, y) \, d\nu(x), \quad \forall y \in \Omega. \]
		
	A positive kernel $G$ on $\Omega\times \Omega$ 
	is said to satisfy the \textit{weak maximum principle (WMP)} with constant $h\ge 1$ if,   for any $\nu \in \mathcal{M}^+(\Omega)$, 
		\begin{equation}\label{wmp}
		 \sup \Big\{\mathbf{G}\nu (x)  \colon x \in \operatorname{supp}(\nu) \Big\} \le M \, \Longrightarrow 
		 \sup \Big\{ \mathbf{G}\nu (x) \colon  x \in \Omega\Big\} \le h \, M,  
		 	\end{equation}
		for any constant $M>0$. When $h=1$,  $G$ is said to satisfy the \textit{strong maximum principle}. It holds for Green's kernels associated with 
		the classical Laplacian, or fractional Laplacian $(-\Delta)^{\alpha}$ in the case $0 < \alpha\le 1$, for all 
		domains $\Omega$ with positive Green's function. The WMP holds for Riesz kernels on $\R^n$ associated 
		with $(-\Delta)^{\alpha}$ 
		in the full range $0<\alpha<\frac{n}{2}$, and more generally for all radially non-increasing kernels 
		on $\R^n$ (see \cite{AH}). 
		
		The WMP also holds for the so-called quasi-metric kernels (see \cite{FNV}, 
		\cite{FV}, \cite{HN}, \cite{QV2}). We say that $d(x,y)\colon \,  \Omega \times \Omega \rightarrow [0, + \infty)$ satisfies the quasimetric triangle inequality with quasimetric constant $\kappa$ if 
		\begin{equation}\label{quasitr} 
			d(x,y) \le \kappa [d(x,z) + d(z, y)], 
		\end{equation}
		for any $x, y, z \in \Omega$.
	We say that $G$ is a \textit{quasimetric} kernel (with quasimetric constant $\kappa>0$) if $G$ is symmetric and $d(x,y) = \frac{1}{G(x,y)}$ satisfies \eqref{quasitr}.

	A kernel $G\colon \Omega \times \Omega \to (0, +\infty]$ is said to be \textit{quasi-symmetric} 
	if  there exists a constant $a$ such that
		\[ a^{-1} G(y,x) \le G(x,y) \le a \, G(y,x), \quad \forall x, y \in \Omega. \] 
		Many kernels associated with elliptic operators are quasi-symmetric and satisfy the WMP 
		(see \cite{An}). 
	
For $0 < q < 1$,  and $\sigma \in \mathcal{M}^+(\Omega)$, we are interested in \textit{positive solutions} $u \in L^r(\sigma)$ ($r>0$) to the integral equation
		\begin{align}\label{int-eq}
			u = \mathbf{G}(u^q \sigma), \quad u>0 \quad  d\sigma-a.e.
		\end{align}
		and \textit{positive supersolutions} $u \in L^r(\sigma)$ to the integral inequality 
		\begin{align}\label{int-sup}
			u \ge \mathbf{G}(u^q \sigma), \quad u>0 \quad d\sigma-a.e.
		\end{align}
	
	In \cite{QV2}, we characterized the existence of positive  solutions $u \in L^q(\Omega, \sigma)$ and 
	$u \in L^q_{\rm loc}(\sigma)$. The latter correspond to the so-called ``very weak''  
	solutions to the sublinear boundary value problem \eqref{lap_eqn} (see \cite{FV}, \cite{MV}). 
		It is easy to see that the condition  $u \in L^q_{\rm loc}(\sigma)$ is necessary for the 
		existence of any positive (super)solution, since otherwise $u\equiv+\infty$ $d \sigma$-a.e. (see \cite{QV2}).   
	
		For a measure $\lambda \in \mathcal{M}^+(\Omega)$, the \textit{energy of $\lambda$} is given by 
	\[
	\mathcal{E}(\lambda) \defeq \int_\Omega \mathbf{G}\lambda \, d\lambda.
	\]
		The notion of energy is closely related to another  major tool of potential theory, the capacity of a set, and the associated equilibrium measure.

	 For a kernel $G \colon \Omega\times \Omega \to (0, +\infty]$, we consider the \textit{Wiener capacity} 
		\begin{equation}\label{wiener} 
			\capa (K)  \defeq \sup \Big\{ \mu(K)\colon  \, \,   \mathbf{G}^*\mu(y) \le 1 \, \, \textrm{on} \, \, \operatorname{supp}(\mu), \, \, \mu \in \mathcal{M}^+(K)\Big \}, 	
		\end{equation}
		defined for compact sets  $K \subset \Omega$.

	The extremal measure $\mu$ for which the supremum in  \eqref{wiener} is attained is called the \textit{equilibrium measure}. Alternatively, capacity can be defined  as a solution to the following extremal problem involving energy:
		\begin{equation}\label{cap-energy} 
		 \capa(K) \defeq \left[ \inf \Big \{ \mathcal{E}(\mu)\colon \, \, \, \mu \in \mathcal{M}^+(K), \quad \mu(K) = 1 \Big\}\right]^{-1}. 	\end{equation}
	
	We say that a property holds \textit{nearly everywhere} (or n.e.) on $K$ when the exceptional set $Z \subset K$ where this property fails has zero capacity, $\capa(Z) = 0$.
	
	We will use the following fundamental theorem \cite{Brelot}, \cite{F}.
	
	\begin{theorem}\label{fuglede_thm} Let $G$ be a positive symmetric kernel on $\Omega\times \Omega$, and let $K\subset \Omega$ a compact set.  The two extremal  problems
		\begin{align*}
			&  \max \Big\{\lambda( K ) \colon \, \,  \mathbf{G}\lambda \le 1 \, \, \text{\textnormal{on}} \, \, \operatorname{supp}(\lambda),  \, \,   \lambda \in \mathcal{M}^+(K) \Big\},  \\
			& \max  \Big\{ 2 \lambda (K) - \mathcal{E}(\lambda) \colon  \, \, \lambda \in \mathcal{M}^+(K)\Big\}, 
		\end{align*}
		always have solutions,  which are precisely the same, and each maximum coincides  with the Wiener capacity $\capa K$.
		The class of all solutions consists of  measures $\lambda \in \mathcal{M}^+(K)$ for which \[ \mathcal{E}(\lambda) = \lambda(\Omega) = \capa(K). \]
		The potential of any solution has the following properties:
		\begin{enumerate}
			\item $\mathbf{G}\lambda(x) \ge 1 $ \text{\textnormal{n.e. in}} $K$, 
			\item $\mathbf{G}\lambda(x) \le 1$ \text{\textnormal{on}} $\operatorname{supp}(\lambda)$, 
			\item $\mathbf{G}\lambda(x) = 1$ $d \lambda\text{\textnormal{-a.e. in}}$ $\Omega$.
		\end{enumerate}
	\end{theorem}
	
		The extremal measure $\lambda$ in Theorem~\ref{fuglede_thm} is  the equilibrium measure for the set $K$.
We observe that since $G$ is a positive kernel, the capacity of all compact sets $K$ is finite.  
(This is true even for non-negative kernels if $G(x, x)>0$ for all $x \in \Omega$; see \cite{F}). 

\section{Weighted norm inequalities, supersolutions, and energy estimates}\label{sec3} 

We begin this section with a proof of Theorem  \ref{strong-thm}. We remark that the ``only if'' part of statement (i) of  Theorem  \ref{strong-thm} is proved without using the assumption that $G$ is quasi-symmetric. Furthermore, the proof of this part works in the case $p=1$ as well. 
		
	\begin{proof}[Proof of Theorem  \ref{strong-thm}] We first prove statement (i). If the $(p, r)$-inequality \eqref{strong-type} 
	holds for $0<r<p$, where $p \ge 1$,  then assuming that 
	$f=(\G \sigma)^{\frac{r}{p-r}} \in L^{p}(\Omega, \sigma)$ and using it as a test function, we deduce 
	$$
	\int_\Omega \left[ \G  \Big( (\G \sigma)^{\frac{r}{p-r}} d \sigma\Big)\right]^r d \sigma \le C^r \ 
 \left[  \int_\Omega (\G \sigma)^{\frac{pr}{p-r}} d \sigma \right]^{\frac{r}{p}}, 
	$$
	where $C$ is the embedding constant in  \eqref{strong-type}. We now use the pointwise inequality 
	\begin{equation}\label{iter-s} 
	 \Big [\G \sigma(x) \Big]^{s} \le s \, h^{s-1} \, \G  \Big( (\G \sigma)^{s-1} d \sigma\Big)(x), \quad x \in \Omega, 
	\end{equation}
	for all $s\ge 1$, established in \cite[Lemma 2.5 and Remark 2.6]{GV2} for non-negative kernels satisfying the WMP with constant $h\ge 1$.
	 Applying \eqref{iter-s}  with $s=\frac{p}{p-r}$, we obtain 
	$$
	\int_\Omega (\G \sigma)^{\frac{pr}{p-r}} d \sigma \le \Big(\frac{p}{p-r}\Big)^r h^\frac{r^2}{p-r} \, C^r  
 \left[  \int_\Omega (\G \sigma)^{\frac{pr}{p-r}} d \sigma \right]^{\frac{r}{p}}. 
	$$
	Since $0<r<p$, this estimate yields 
	$$
	\int_\Omega (\G \sigma)^{\frac{pr}{p-r}} d \sigma\le  \Big(\frac{p}{p-r}\Big)^{\frac{pr}{p-r}} h^\frac{pr^2}{(p-r)^2} \,  C^{\frac{pr}{p-r}}. 
	$$
	The extra assumption that 	$f=(\G \sigma)^{\frac{r}{p-r}} \in L^{p}(\Omega, \sigma)$ is easy to remove 
		by using $\chi_K \, f$ in place of $f$, where $K$ is a compact subset of $\Omega$ on which 
	$\G \sigma(x) \le n$, and then letting $n \to +\infty$  (see details in \cite{QV2}). 
	
	In the opposite direction, suppose that \eqref{r-p} holds for $0<r<p$ and $p>1$. 
	Without loss of generality we may assume that $f \ge 0$. By H\"older's 
	inequality,
	\begin{align*}
	\int_\Omega [\G( f d \sigma)]^r d \sigma & = \int_{\Omega}\left [\frac{\G( f d \sigma)}{\G \sigma}\right]^r 
	(\G \sigma)^r d \sigma \\ & \le \left[\int_{\Omega}\left (\frac{\G( f d \sigma)}{\G \sigma}\right)^p d \sigma\right]^{\frac{r}{p}}
	\left[\int_\Omega(\G \sigma)^{\frac{pr}{p-r}} d \sigma \right]^{1-\frac{r}{p}}.
		\end{align*}
	
	We next sketch a proof of a $(1,1)$-weak type estimate obtained  in a more general context  in \cite[Lemma 5.10]{QV2}: 
		\begin{equation}\label{weak-1-1} 
	\left \Vert \frac{\G( f d \sigma)}{\G \sigma} \right \Vert_{L^{1, \infty} (\Omega, d \sigma)} \le 
	c \, ||f||_{L^1(\Omega, d \sigma)}, 
		\end{equation} 
		where  $c=c(h, a)$ depends only on the constants $h\ge 1$ in the weak maximum principle, and 
		$a>0$ in the quasi-symmetry condition. 
		
		Since $G$ is quasi-symmetric, we can assume without loss of generality that it is symmetric by replacing $G$ with $\frac{1}{2}(G+ G^*)$. 
		Let  $E_t = \{ x \in \Omega\colon \frac{\G( f d \sigma)}{\mathbf{G}\sigma} 
		(x)  > t \}$, where $t>0$.  
				For an arbitrary compact set $K \subset E_t$, we denote by  $\mu \in \mathcal{M}^+(K)$ an equilibrium measure on $K$ (see Sec. \ref{background} above)  such that
			$\mathbf{G}\mu \ge 1$ n.e. on $K$ and $\mathbf{G}\mu \le 1$ on $\operatorname{supp} (\mu)$. 
			
				It is easy to see that  in fact 
				\begin{equation}\label{n.e.}
				\mathbf{G}\mu \ge 1 \quad d \sigma-\textrm{a.e. on} \,\,  K. 
				\end{equation} 
				Indeed, from \eqref{r-p} it follows that $\mathbf{G} \sigma<+\infty$ $d \sigma$-a.e. Since $\mathbf{G}\mu \ge 1$ n.e. on $K$, 
			 the set $Z = \{ x \in K\colon \,  \mathbf{G}\mu(x) < 1 \}$ has zero capacity, and 
			 consequently,  
			\begin{align*}
		\sigma (Z) &=\sigma (\{x\in Z \colon \,\,  \mathbf{G}\sigma(x)<+\infty\})
		\\& \le \sum_{n=1}^{+\infty} \sigma(\{x\in Z \colon \, \mathbf{G}\sigma(x) \le n\}) \\
		& \le  \sum_{n=1}^{+\infty} n \,  \operatorname{cap} (\{x \in Z \colon \, \mathbf{G}\sigma(x) \le n\}) =0.   
		\end{align*}
		Thus, $\sigma(Z)=0$, which proves \eqref{n.e.}.

	 Since $\mathbf{G}\mu \le 1$ on $\operatorname{supp}(\mu)$, 
		it follows that $\mathbf{G}\mu \le h$ on $\Omega$ by the WMP. From this and \eqref{n.e.}, using Fubini's theorem,  
		we deduce 
			\begin{align*}
				\sigma(K) & \le \int_K \G\mu \, d\sigma = \int_K \mathbf{G}\sigma_K \, d\mu \\
					&\le  \int_K \frac{\G( f d \sigma)}{t} \, d\mu = \frac{1}{t} \int_K \mathbf{G} \mu \, f \, d\sigma \\  &\le \frac{1}{t} \int_\Omega h \, f \, d\sigma = \frac{h}{t} \,  ||f||_{L^1(\Omega, \sigma)}. 
			\end{align*}	
		Taking the supremum over all  $K\subset E_t$, we obtain 
			\[ \sigma(E_t) \le \frac{h}{t} \,  ||f||_{L^1(\Omega, \sigma)}, \]
		which proves \eqref{weak-1-1}.

The corresponding $L^\infty$ estimate is obvious:
\begin{equation*}
	\left \Vert \frac{\G( f d \sigma)}{\G \sigma} \right \Vert_{L^{\infty} (\Omega, d \sigma)} \le 
	 ||f||_{L^\infty(\Omega, d \sigma)}. 
		\end{equation*}
Thus, for $1<p<+\infty$, by the Marcinkiewicz  interpolation theorem we obtain 
	$$
	\left \Vert \frac{\G( f d \sigma)}{\G \sigma}  \right \Vert_{L^{p} (\Omega, d \sigma)} \le 
	C \, ||f||_{L^p(\Omega, d \sigma)},  
	$$
	for all $f \in L^p(\Omega, d \sigma)$. Hence, combining the preceding estimates, we deduce 
	$$
	\int_\Omega [\G( f d \sigma)]^r d \sigma \le	C \, ||f||^{r}_{L^p(\Omega, \sigma)} 
	\left[\int_\Omega(\G \sigma)^{\frac{pr}{p-r}} d \sigma \right]^{1-\frac{r}{p}}.
		$$
		This proves statement (i). 
		
		We now prove statement (ii). Let $0<q<1$. Suppose there exists a positive supersolution 
		 $u \in L^r(\Omega, \sigma)$ with $r>q$. As shown in \cite[Corollary 3.6]{GV2}, 
		if $G$  satisfies 
		the WMP, then 
		any nontrivial supersolution $u$ satisfies the global pointwise bound 
		\begin{equation}\label{lower-est}
		u(x) \ge (1-q)^{\frac{1}{1-q}}  h^{-\frac{q}{1-q}}  \, [ \G \sigma (x)]^{\frac{1}{1-q}} \quad  d\sigma-\textrm{a.e.} 
		\end{equation}
		Thus,  
		\eqref{r-q} holds. 
		
		Conversely,  by statement (i), \eqref{r-q} with $r>q$ 
		implies the $(p, r)$-inequality \eqref{strong-type} with 
		$p=\frac{r}{q}$. Letting $u_0 = c \, [ \G \sigma (x)]^{\frac{1}{1-q}}$ where $c>0$ is a 
		positive constant, we get a sequence of iterations 
		$$
		u_{j+1} = \G(u_j^q \, d \sigma), \quad j=0, 1, \ldots, 
		$$
		where by induction we see that $u_{j+1} \ge u_j$, provided the constant $c$ is small enough.  
		Here the initial step $u_1\ge u_0$ follows from \eqref{iter-s} with $s=\frac{1}{1-q}$, since 
		$$
		u_1 = \G(u_0^q d \sigma)=c^q \, \G \Big[(\G \sigma)^{\frac{q}{1-q}}d \sigma\Big] \ge  c \, [ \G \sigma (x)]^{\frac{1}{1-q}}=u_0,
		$$ 
		for an appropriate choice of $c=c(q, h, a)$. By  \eqref{strong-type} with 
		$p=\frac{r}{q}$ and $f=u_j$, we have by induction, 
		\begin{align*}
		||u_{j+1}||_{L^r(\Omega, \sigma)} =
		  \Big \Vert\G(u_j^q \, d \sigma)\Big \Vert_{L^r(\Omega, \sigma)} 
		  \le C \, \Vert u_j\Vert^{q}_{L^r(\Omega, \sigma)}<+\infty. 
		   		\end{align*}
		Since $0<q<1$ and $u_j \le u_{j+1}$, it follows that 
		$$
		||u_{j+1}||_{L^r(\Omega, \sigma)} \le C(r, q, h, a), \quad j=0, 1, \ldots. 
		$$
		Using the monotone convergence theorem, we obtain a positive  solution 
		$$u=\lim_{j\to \infty} u_j, \quad u\in L^r(\Omega, \sigma).$$ 
		\end{proof}

	Theorem \ref{strong-thm}   makes use of  energy conditions of the type 
		\begin{equation} \label{energy}
			\int_{\Omega} (\G \sigma)^s d \sigma< \infty,  
		\end{equation}
		for some $s>0$. 
	Note that when $s = 1$, this gives the energy $\mathcal{E}(\sigma)$ introduced above.
		
	In the next lemma, we deduce \eqref{energy} for 
	 $s=\frac{r}{1-q}$ provided there exists a positive supersolution 
	 $u \in L^r(\Omega, \sigma)$ to  \eqref{super_sol}, for non-negative, quasi-symmetric kenels $G$, 
	without assuming that \eqref{strong-type} holds, or that $G$ satisfies the WMP. In the special case 
	$r=q$ it was proved in \cite[Lemma 5.1]{QV2}.

	\begin{lemma}\label{lemma2}
		Let $\sigma \in \mathcal{M}^+(\Omega)$, and let $0<q<1$. Suppose $G$ is a non-negative quasi-symmetric kernel on $\Omega\times \Omega$. Suppose there is a positive supersolution $u \in L^r(\Omega, \sigma)$ 
		$(r>0)$, i.e., $\mathbf{G} (u^q d \sigma)\le u$ $d \sigma$-a.e. Let $0 < q \le 1-r^2$.  Then 
			\begin{equation}\label{est1}
			\int_{\Omega} (\G \sigma)^{\frac{r}{1-q}} d \sigma \le a^{\frac{rq}{(1-q)(1-r+q)}} \int_{\Omega} 
			 u^{r} 
			d \sigma <+\infty,   
			\end{equation}
			where $a$ is the quasi-symmetry constant of $G$. 
		\end{lemma}

	\begin{proof}
		Suppose $u\in L^r(\Omega, \sigma)$, where $0<r<1$,  is a  positive supersolution. 
		Let $\gamma\ge 1$.		By H\"older's inequality with exponents $\gamma$ and $\gamma'=\frac{\gamma}{\gamma-1}$, we estimate 
		\begin{equation*}
			\begin{aligned} 
			\G \sigma (x) & = \int_{\Omega}  u^{\frac{q}{\gamma}}  u^{-\frac{q}{\gamma}} G(x, y) \, d \sigma(y)\\& \le \left[ \G(u^q 
			d \sigma(x)\right]^{\frac{1}{\gamma}}  \left[\G(u^{-\frac{q}{\gamma-1}} d \sigma)(x)\right]^{\frac{1}{\gamma'}} \\&
			\le [u(x)]^{\frac{1}{\gamma}}  \left[\G(u^{-\frac{q}{\gamma-1}} d \sigma)(x)\right]^{\frac{1}{\gamma'}} . 
			\end{aligned}
		\end{equation*}
		Let $\gamma = 1+\frac{q}{1-r}$, where $0<r \le 1-q^2$. Then $\frac{(1-q)\gamma'}{r}\ge 1$. 
		Using the preceding inequality, along with H\"older's inequality with the conjugate exponents 
		$$ \frac{(1-q)(1-r+q)}{1-r-q^2} >1 \quad  \textrm{and} \quad 
		\frac{(1-q)(1-r+q)}{rq} \ge 1,$$ 
		and Fubini's theorem, we estimate

		\begin{equation*}
			\begin{aligned} 
				\int_\Omega (\G \sigma)^{\frac{r}{1-q}} d \sigma & \le \int_{\Omega} u^{\frac{r}{(1-q)\gamma}}   
				\left[\G(u^{r-1} d \sigma)\right]^{\frac{r}{(1-q)\gamma'}} 
				 d \sigma   \\&=  \int_{\Omega} u^{\frac{r(1-r-q^2)}{(1-q)(1-r +q)}}   
				\left[ u^q \, \G(u^{r-1} d \sigma)\right]^{\frac{rq}{(1-q)(1-r+q)}} 
				 d \sigma 				  \\& 
				 \le \left[ \int_{\Omega}  u^{r} d \sigma\right]^{\frac{1-r-q^2}{(1-q)(1-r +q)}}  
				 \left[\int_{\Omega} \G(u^{r-1} d \sigma) \,  u^q \,d \sigma\right]^{\frac{rq}{(1-q)(1-r+q)}} 
				 \\
				& = 
				\left[ \int_{\Omega}  u^{r} d \sigma\right]^{\frac{1-r-q^2}{(1-q)(1-r +q)}}  
				\left[ \int_{\Omega} \G^{*}(u^q d \sigma) u^{r-1} d \sigma\right]^{\frac{rq}{(1-q)(1-r+q)}}  
				\\& 
				\le a^{\frac{rq}{(1-q)(1-r+q)}}   \left[ \int_{\Omega}  u^{r} d \sigma\right]^{\frac{1-r-q^2+rq}{(1-q)(1-r +q)}}.
			\end{aligned}
		\end{equation*}
		In the last estimate we used the inequality $ \G^{*}(u^q d \sigma) \le a \, u$. Since 
		$1-r-q^2+rq=(1-q)(1-r +q)$, this completes the proof of \eqref{est1}. 
	\end{proof} 
	
We next show that, for general non-negative kernels $G$, the $(p, r)$-weighted 
norm inequality \eqref{strong-type} with $p=\frac{r}{q}\ge 1$ yields 
the existence of a supersolution $u \in L^r(\Omega, \sigma)$ to 
\eqref{super_sol}. This is deduced from Gagliardo's lemma \cite{G} 
(see also \cite{Szeptycki}), as in the special case $r=q$ in \cite{QV2}.  
	  
It  will be convenient for us to construct a measurable function $\phi$ such that 
	\begin{equation}\label{phi}	
		0<  [\mathbf{G}(\phi \, d\sigma)]^q \le \phi <+\infty \quad d \sigma-\textnormal{a.e.},
		\end{equation} 
		for $0<q<1$. 
Clearly, 
		if $\phi$ satisfies the above estimate, then $u=\phi^{\frac{1}{q}}$ satisfies \eqref{super_sol}. 
		Moreover, $u \in L^r(\Omega, \sigma)$ if $\phi \in L^p(\Omega, \sigma)$, where $p=\frac{r}{q}\ge 1$.

    We recall that a convex cone $P\subset B$ is \textit{ strictly convex at the origin} if, for any $\phi, \psi \in P$,   
      $\alpha \phi + \beta \psi = 0$  implies $\phi = \psi = 0$, for any $\alpha, \beta > 0$ such that $\alpha + \beta = 1$.

    \begin{lemma}[Gagliardo \cite{G}]
    	\label{gag_lemma}
        Let $B$ be a Banach space, and let $P \subset B$ be a convex cone which is strictly convex at the origin and such that if $ (\phi_n) \subset P$, $\phi_{n+1} - \phi_n \in P$, and $\Vert \phi_n \Vert \le M$ for all $n = 1, 2, \dots$, then there exists $\phi \in P$ so that $\Vert \phi_n - \phi \Vert \rightarrow 0$.
        
          Let $S\colon \,  P \rightarrow P$ be a continuous mapping with the following properties:
        \begin{enumerate}
        	\item For $\phi, \psi \in P$, such that $\phi - \psi \in P$, we have $S\phi - S\psi \in P$.
        	\item If $\Vert \phi \Vert \le 1$ and $\phi \in P$, then $\Vert Su \Vert \le 1$.
        \end{enumerate}
        	Then for every $\lambda > 0$ there exists $\phi \in P$ so that $(1 + \lambda)\phi - S\phi \in P$ and $0 < \Vert \phi \Vert \le 1$. Moreover, for every $\psi \in P$ such that $0<\Vert \psi \Vert_B \le \frac{\lambda}{1+\lambda}$, $\phi$ can 
	be chosen so that $\phi =\psi + \frac{1}{1+\lambda} S \phi$. 
    \end{lemma}
    
    We will apply this lemma to $B= L^p(\sigma)$, $p\ge 1$, and the cone of non-negative functions $P$ 
    in $B$. In this case obviously one can ensure that $\phi>0$ $d \sigma$-a.e.  
        
	\begin{lemma}
		\label{phi_soln}
		Let $(\Omega, \sigma)$ be a sigma-finite measure space, and let $G$ be a non-negative kernel on $\Omega\times \Omega$. Let $0<r<+\infty$  and $0<q<1$. 
		Suppose  \eqref{strong-type} holds for $p=\frac{r}{q}\ge 1$ 
		with an embedding constant $C=\varkappa>0$. 
		Then, for every $\lambda > 0$, there is a positive  $\phi \in L^p(\sigma)$ satisfying \eqref{phi} so that 
		$$\Vert \phi \Vert_{L^p(\sigma)} \le (1+\lambda)^{\frac{1}{1-q}} \varkappa^{\frac{q}{1-q}}.$$
	\end{lemma}
	\begin{proof}
		The supersolution $\phi$ is constructed using Lemma \ref{gag_lemma}.
		Define $S\colon \,  L^p(\sigma) \rightarrow L^p(\sigma)$  by 
		\[S \phi \defeq \Big[\frac{1}{\varkappa^q} \, \mathbf{G}(\phi \, d\sigma)\Big]^q,\]
		for all $\phi \in L^p(\sigma)$, $\phi \ge 0$.
		Inequality \eqref{strong-type} gives that $S$ is a bounded continuous operator.
		In fact, by \eqref{strong-type} we see  that if $\Vert \phi \Vert_{ L^p(\sigma)} \le 1$, then
		\begin{align*}
			\Vert S(\phi) \Vert^p_{L^p(\sigma)}
				&= \frac{1}{\varkappa^r} \int_\Omega [\mathbf{G}(\phi \sigma)]^r \, d\sigma \\
				&= \frac{1}{\varkappa^r} \varkappa^r \left(\int_\Omega \phi^p \, d\sigma\right)^q\le 1.
		\end{align*}
		Therefore, by  Lemma~\ref{gag_lemma}, there exists $\phi \in L^p(\sigma)$ such that 
			\[ (1+\lambda) \phi \ge \frac{1}{\varkappa^q} [\mathbf{G}(\phi \sigma)]^q, \]
			$\Vert \phi \Vert_{L^p(\sigma)} \le 1$, and $\phi > 0$ $d \sigma$-a.e.
		Setting $\phi_0 = c \,  \phi$, where 
			\[ c= \left[ \frac{1}{(1 + \lambda) \varkappa^q} \right]^{\frac{1}{1-q}}, \]
			we deduce that $\phi > 0$ $d \sigma$-a.e., and 
			\[ \phi_0 \ge \mathbf{G}(\phi_0 \sigma)^q, \quad \Vert \phi_0 \Vert_{L^p(\sigma)} \le (1 + \lambda)^{\frac{1}{1-q}} \varkappa^{\frac{q}{1-q}}.\]
	\end{proof}

Remark \ref{rm2}  follows immediately from Lemma \ref{phi_soln}.

\begin{remark}\label{rm4} {\rm For $p=\frac{r}{q}$, a counterexample 
in \cite{QV2} demonstrates that, without the WMP, the existence of a supersolution $u\in L^r(\Omega, \sigma)$ to \eqref{super_sol} in the case $r=q$ does not imply the
$(p, r)$-weighted norm inequality \eqref{strong-type}, even for positive symmetric kernels $G$. A slight modification of that counterexample shows that the same is true in the case $r>q$ as well. 
}
\end{remark}

	\section{A counterexample in the end-point case $p=1$}
	\label{sec2}
	
	In the case $p=1$, $0<q<1$,  the $(1, q)$-weighted norm inequality  \eqref{strong-type} with $r=q$ 
follows from a similar inequality for the space of measures $\mathcal{M}^{+}(\Omega)$ in place of 
$L^1(\Omega, \sigma)$, 
\begin{equation}\label{strong-meas}
\Vert \mathbf{G} \nu \Vert_{L^q(\Omega, \sigma)} \le C \, ||\nu||, \quad \forall \nu \in \mathcal{M}^{+}(\Omega),
\end{equation}
where $||\nu||=\nu(\Omega)$. This inequality was shown in \cite{QV2} to be equivalent to the existence of a positive 
supersolution $u \in L^q(\Omega, \sigma)$ to \eqref{super_sol} for quasi-symmetric kernels $G$ satisfying 
the WMP. In this case, \eqref{strong-meas} is equivalent to \eqref{strong-type}  with $r=q$ and $p=1$ in view of Lemma \ref{phi_soln}. 

 However, a characterization 
	of  \eqref{strong-meas}, or \eqref{strong-type} with $r=q$ and $p=1$, in terms of the energy estimate \eqref{r-q} with $r=q$ is not available, contrary to the case $r>q$:  the condition  
	\begin{equation}\label{q-energy}
	\int_\Omega (\mathbf{G} \sigma)^{\frac{q}{1-q}} d \sigma<+\infty
	\end{equation} 
	is not sufficient for  \eqref{strong-meas}.  
	
	On the other hand, it is not difficult to see that  \eqref{strong-meas} holds for all $\nu \in \mathcal{M}^{+}(\Omega)$ 
	if and only if it holds for all finite linear combinations of point masses, $\nu =\sum_{j=1}^n a_j \, \delta_{x_j}$, $a_j>0$. It had been conjectured that, for $0<q<1$, condition  \eqref{q-energy} combined with \eqref{strong-meas} 
	for single point masses $\nu=\delta_{x}$, i.e., 
	\begin{equation}\label{point-mass}
	\int_\Omega G (x,y)^q \, d \sigma(y) \le C <+\infty, \quad \forall x \in \Omega,
	\end{equation} 
was not only necessary, but also sufficient 
	for \eqref{strong-meas}. (Notice that in the case $q\ge 1$  \eqref{point-mass} is obviously necessary and sufficient 
	for \eqref{strong-meas}; see \cite{QV2}.) 
	
	In this section, we give a counterexample to this conjecture 
	for Riesz potentials on $\R^n$, 
	$$
	\mathbf{I}_{2 \alpha} \nu(x) = \int_{\R^n} \frac{d \nu(y)}{|x-y|^{n-2 \alpha}}, \quad x \in \R^n,
	$$
	where $\nu \in \mathcal{M}^{+}(\R^n)$, and $0<2\alpha<n$. Clearly, Riesz kernels 
	$|x-y|^{2 \alpha-n}$ are symmetric, and satisfy the WMP.

	Suppose  $0<q<1$, $n \ge 1$, and  $0< 2 \alpha < n$. We construct 
	$\sigma \in \M$ such that 
	\begin{equation}\label{A}
	  \E(\sigma)=\E_{\alpha, q} (\sigma)\defeq \int_{\R^n} \Big (\mathbf{I}_{2 \alpha} \sigma\Big)^{\frac{q}{1-q}} d \sigma < +\infty,
	\end{equation}
	and 
	\begin{equation}\label{B}
	\mathcal{K} (\sigma)=\mathcal{K}_{\alpha, q} (\sigma)\defeq \sup_{x \in \R^n} \int_{\R^n} \frac{d \sigma(y)}{|x-y|^{(n-2\alpha)q}}< +\infty,
	\end{equation}
	but 
	\begin{equation}\label{C}
	\kappa(\sigma) = \kappa(\sigma)_{\alpha, q}\defeq \sup \left\{ \frac{|| \I_{2 \alpha} \nu||_{L^q(\sigma)}}{||\nu||_{\M}} \colon \, \, \nu \in \M, \, \,  \nu \not=0\right\} = +\infty.
	\end{equation}

	In other words, we need to construct a measure $\sigma$ such that $ \mathcal{E}(\sigma)<+\infty$ (in the special case $q=\frac{1}{2}$  this means that 
	$\sigma$ has finite energy), and 
	\eqref{C} holds for all $\delta$-functions $\nu = \delta_x$ ($x \in \R^n$), but \eqref{C} fails for a linear combination of $\delta$-functions 
	\begin{equation}\label{delta}
	\nu =\sum_{j=1}^{\infty} a_j \, \delta_{x_j}, \quad \textrm{where} \quad \sum_{j=1}^{\infty} a_j < +\infty, \quad a_j>0.
	\end{equation}
	 We will use a modification of the example considered in \cite{CV2} for other purposes.

	We will need the following lemma and its corollary in the radially symmetric case (see \cite{CV2}). 
	\begin{lemma}\label{lemvar} Let $0<q<1$ and $0<2 \alpha < n$. If   $d \sigma = \sigma(|x|) \, dx$ is radially symmetric,  
	then $\kappa(\sigma)<+\infty$ if and only if  $\mathcal{K}(\sigma)< +\infty$. Moreover,  there exists  a constant $c=c(q, \alpha, n)>0$ 
	such that  $\kappa(\sigma)$  satisfies 
	\begin{equation}\label{const-a}
	\mathcal{K}(\sigma) \le \kappa(\sigma)^q \le c \, \mathcal{K}(\sigma),  
	\end{equation}
	where in the this case 
	\begin{equation}\label{rad-const}
	\mathcal{K}(\sigma) =  \int_{\R^n} \frac{d \sigma(y)}{|y|^{(n-2\alpha)q}}. 
	\end{equation}
	\end{lemma}
	
	\begin{remark}{\rm For radially symmetric $\sigma$, condition $\mathcal{K}(\sigma)< +\infty$ 
	is equivalent to $\sigma\in L^{\frac{1}{1-q},1}(\R^n, \sigma)$, 
	which is necessary and sufficient for  \eqref{strong-meas} in this case; see \cite{QV1}, \cite{QV2}. 
	Here $L^{s,1}(\R^n, \sigma)$ denotes the corresponding Lorentz space with respect to the measure 
	$\sigma$.}  
	\end{remark}
	
	\begin{cor}\label{cor-c} Let $\sigma_{R, \gamma} = \chi_{B(0, R)} |x|^{-\gamma}$, where 
	$0\le \gamma<n-q(n-2 \alpha)$ and $R>0$. Then 
	\begin{equation}\label{rad-K}
	\mathcal{K}(\sigma) =  \frac{ \omega_n \, R^{n-\gamma - q(n-2 \alpha)}} {n-\gamma - q(n-2 \alpha)},
	\end{equation}
	and 
	\begin{equation}\label{rad-kappa}
	\frac{\omega_n}{n-\gamma - q(n-2 \alpha)}     \le \frac{\kappa(\sigma_{R, \gamma})^q} {R^{{n-\gamma-q(n-2 \alpha)}}}
	 \le \frac{c}{n-\gamma - q(n-2 \alpha)},  
	\end{equation}
	where $c=c(q, \alpha, n)$, and $\omega_n= |S^{n-1}|$ is the surface area of the unit sphere. 
	\end{cor}
	
	Let 
	\begin{equation}\label{sigma}
	\sigma = \sum_{k=1}^\infty c_k \sigma_k, 
	\end{equation}
	where 
	\begin{equation}\label{r-k}
	\sigma_k = \sigma_{R_k, \gamma_k} (x+x_k), \quad R_k=|x_k|=k, \quad \gamma_k = n-q(n-2 \alpha) - \epsilon_k,
	\end{equation}
	 and the positive scalars $c_k$, 
	$\epsilon_k$ 
	are picked so that $\sum_{k=1}^\infty c_k < \infty$, $\epsilon_k \to 0$, and 
	$0<\gamma_k <n$.  Notice that $\gamma_k \to n-q(n-2 \alpha)$ as $k \to \infty$, which  
	is a critical exponent for the inequality \eqref{sobolev} (with $\sigma_k$ in place of $\sigma$) discussed 
	below.

	More precisely, for $0<q<1$ and $0<\delta<+\infty$, we set 
	\begin{equation}\label{choice}
	 a_k = \frac{1}{k (\log (k+1))^{\frac{1}{q}}},  \quad c_k=\frac{1}{k^{2-q+\delta}},  \quad \epsilon_k = \frac{1}{k^{1+\delta}}, \quad k=1, 2, \ldots , 
	\end{equation}
	so that 
	\begin{equation}\label{sums} 
	\sum_{k=1}^{+\infty} a_k<+\infty, \quad \sup_{k\ge 1} \frac{c_k}{\epsilon_k} < +\infty,  \quad \sum_{k=1}^{+\infty} \frac{c_k}{\epsilon_k^{1-q}}<+\infty,  \quad \textrm{but} \quad \sum_{k=1}^{+\infty} 
	\frac{c_k \, a_k^q}{\epsilon_k}=+\infty. 
	\end{equation}

	We first verify condition \eqref{A}. Notice that 
	 \begin{equation}\label{const}
	c_1 \, A \le [\E_{\alpha, q}(\sigma)]^{1-q} \le c_2 \, A, 
	 \end{equation}
	where $A$ is the least constant in the inequality (see \cite{COV06}; \cite{CV1}, Lemma 3.3) 
	\begin{equation}\label{sobolev}
	\int_{\R^n} | \I_{\alpha} f |^{1+q} \, d \sigma  \le A \, || f ||^{1+q}_{L^2(dx)}, \quad \text{for all} \, \, f \in L^2(\R^n, dx),  
	 \end{equation}
	 or, equivalently,  
	 \begin{equation}\label{sobolev2}
	\int_{\R^n} | \I_{2 \alpha} (g d \sigma) |^{1+q} \, d \sigma  \le A^{2} \, || g||^{1+q}_{L^{\frac{1+q}{q}}(d \sigma)}, \quad \text{for all} \, \, g \in L^2(\R^n, \sigma),   
	 \end{equation}
	 where the constants of equivalence $c_1$, $c_2$ in \eqref{const}  depend 
	only on $\alpha$, $q$, and $n$.  
	
	Consequently, $[\E_{\alpha, q}(\sigma)]^{1-q}$ is equivalent to a \textit{norm} on a subset of $\M$, so that 
	\begin{equation}\label{norm}
	\Big[\E_{\alpha, q}\Big(\sum_k \sigma_k\Big)\Big]^{1-q}\le c \, \sum_k \Big[\E_{\alpha, q}(\sigma_k)\Big]^{1-q}, 
	 \end{equation}
	where $c=c(\alpha, q, n)$ is a positive constant which depends  only on $\alpha$, $q$, and $n$. 
	
	We claim that, 
	\begin{equation}\label{norm-est}
	\E_{\alpha, q}(\sigma_k) \le  \frac{C \, {R_k}^{\frac{\epsilon_k}{1-q}}}{\epsilon_k}, 
	\quad k=1, 2, \ldots , 
	 \end{equation}
	 where $C=C(\alpha, q, n)$. 
	 
	 Indeed, by the semigroup property of Riesz kernels, 
	 \begin{equation*}
	 \begin{aligned}
	 \I_{2 \alpha} \sigma_k (x) & = c(\alpha, n) \int_{B(0, R_k)} \frac{d t}{|x-t|^{n- 2 \alpha} 
	 |t+x_k|^{\gamma_k} }\\ & \le c(\alpha, n) \int_{\R^n} \frac{d t}{|x-t|^{n- 2 \alpha} 
	 |t+x_k|^{\gamma_k} }
	 =  c \, |x+x_k|^{2 \alpha-\gamma_k}, 
	 \end{aligned}
	  \end{equation*}
	 where $c = c(n, 2 \alpha+n -\gamma_k)$ remains bounded by a constant $C(\alpha, q, n)$ as $k \to +\infty$, since $\lim_{k\to +\infty}  (2 \alpha+n -\gamma_k) = 2 \alpha +q(n-2 \alpha)<n$. 
	 
	 Notice that $(\gamma_k - 2 \alpha)\frac{q}{1-q} + \gamma_k = n -\frac{\epsilon_k}{1-q}$. 
	 Hence, by the preceding estimate, 
	 \begin{equation*}
	 \begin{aligned}
	\E_{\alpha, q}(\sigma_k) & = \int_{\R^n} \Big(\I_{2 \alpha} \sigma_k\Big)^{\frac{q}{1-q}} d \sigma_k 
	\\ & \le c^{\frac{q}{1-q}} \int_{|x+x_k|<R_k} \frac{dx}{|x+x_k|^{n-\frac{\epsilon_k}{1-q}}} \\& = c^{\frac{q}{1-q}}  \omega_n \int_{0}^{R_k} r^{\frac{\epsilon_k}{1-q}-1} dr
	\\& \le   \frac{C(\alpha, q, n) \, {R_k}^{\frac{\epsilon_k}{1-q}}}{\epsilon_k},
	 \end{aligned}
	  \end{equation*}
	which proves \eqref{norm-est}.

	It follows from \eqref{norm} and the preceding estimate that, for $\sigma$ defined by \eqref{sigma}, 
	\begin{equation}\label{sum} 
	 \begin{aligned}
	\Big[\E_{\alpha, q}(\sigma)\Big]^{1-q} 
	 & \le c(\alpha, q, n) \,  \sum_k c_k \Big[\E_{\alpha, q}(\sigma_k)\Big]^{1-q}\\
	& \le c(\alpha, q, n)  \, C(\alpha, q, n)^{1-q}  \, \sum_k \frac{c_k \, R_k^{\epsilon_k}}{\epsilon_k^{1-q}}<+\infty, 
	\end{aligned}
	\end{equation}
	by \eqref{sums}, since obviously $\sup_{k \ge 1} R_{k}^{\epsilon_k} < +\infty$ by \eqref{choice}. This proves \eqref{A}.

	 To prove \eqref{B}, we will need the following lemma.

	 \begin{lemma}\label{lemma} 
	 Let  $R>0$, $0< \beta<n$, and  $0<\epsilon<n-\beta$. For $\gamma = n-\beta-\epsilon>0$, we have 
	\begin{equation}\label{mathcal-Kb} 
	\phi_{R, \gamma} (x) \defeq \int_{|t|<R}  \frac{dt }{|x-t|^{\beta} |t|^{\gamma}} \approx 
	\left\{ \begin{array}{ll}
	 \frac{R^{\epsilon}-|x|^{\epsilon}}{\epsilon} & \quad \textrm{ if ~~}   |x| \le  \frac{R}{2},\\ R^{\epsilon} \Big(\frac{R}{|x|}\Big)^{\beta} & \quad \textrm{ if ~~}  |x| >  
	\frac{R}{2},
	\end{array} \right.
	 \end{equation}
	 where the constants of equivalence depend only on  $\beta$ and $n$. 
	 \end{lemma} 
	 
	 \begin{proof} Suppose first that  $|x| > \frac{R}{2}$. Then 
	 \begin{equation*}
	 \begin{aligned}
	\phi_{R, \gamma} (x) &  =  \int_{|t|<\frac{R}{4}}  \frac{dt }{|x-t|^{\beta} |t|^{\gamma}} + 
	 \int_{\frac{R}{4}<|t|<R}  \frac{dt }{|x-t|^{\beta}  |t|^{\gamma}}
	 \\ & \defeq I + II.  
	  \end{aligned}
	\end{equation*}
	Clearly, in the first integral $ \frac{|x|}{2}\le |x-t|\le \frac{3 |x|}{2} $,  and so $I$ is bounded above and below by 
	$$
	  \frac{\omega_n \, c(\beta)}{|x|^{\beta}} \int_{0}^{R} r^{n-1-\gamma} dr = \frac{c(\beta, n) R^{n-\gamma}}{|x|^{\beta}}. 
	$$
	 To estimate the second term, notice that, for $|x|>2 R$ and $|t|<R$,  we have 
	 $|x-t|>\frac{|x|}{2}$, so that 
	  \begin{equation*}
	 II\le  \frac{c(\beta, n)}{ R^{\gamma} \, |x|^{\beta}} \int_{\frac{R}{4}<|t|<R} dt =
	  \frac{c(\beta, n)R^{n-\gamma}} {|x|^{\beta}}.   
	\end{equation*}
	For $\frac{R}{2}< |x|< 2 R$ and $|t|<R$, we have $|x-t|< 3R$, and consequently 
	 \begin{equation*}
	 \begin{aligned}
	II & \le 
	  \frac{c(\beta, n)}{R^{\gamma}}  \int_{|x-t|< 3 R}  \frac{dt }{|x-t|^{\beta}} \\ &
	 =  \frac{\omega_n c(\beta, n)}{R^{\gamma}} \int_{0}^{3R} r^{n-1-\beta} dr 
	\\ &  = C(\beta, n) R^{n -\beta-\gamma} \\ & \le  \frac{C(\beta, n) R^{n-\gamma}}{|x|^{\beta}}. 
	  \end{aligned}
	\end{equation*}
	Thus, $I I  \le c(n, \beta) \, I$, which proves \eqref{mathcal-Kb} in the case $|x|\ge  \frac{R}{2}$. 
	
	Suppose now that $|x| \le  \frac{R}{2}$. Then 
	 \begin{equation*}
	 \begin{aligned}
	\phi_{R, \gamma} (x) &  =  \int_{|t|<\frac{|x|}{2}}  \frac{dt }{|x-t|^{\beta} |t|^{\gamma}} + 
	  \int_{\frac{|x|}{2}<|t|<2 |x|}  \frac{dt }{|x-t|^{\beta}  |t|^{\gamma}} +  \int_{2 |x|<|t|<R}  \frac{dt }{|x-t|^{\beta}  |t|^{\gamma}} 
	 \\ & \defeq I I I+ IV + V.  
	  \end{aligned}
	\end{equation*}
	
	Clearly, in the first integral $\frac{|x|}{2} < |x-t|< \frac{3|x|}{2}$, and so 
	$III$ is bounded above and below by 
	$$
	\frac{c(\beta)}{|x|^\beta}  \int_{|t|<\frac{|x|}{2}}  \frac{dt }{|t|^{\gamma}} 
	= \frac{\omega_n c(\beta)}{|x|^\beta}  \int_{0}^{\frac{|x|}{2}} r^{n-1-\gamma} dr 
	= \frac{c(\beta, n)}{(n-\gamma) 2^{n-\gamma}} |x|^{\epsilon}. 
	$$
	 
	 The second integral $IV$ is bounded above and below by 
	 $$
	  \frac{c(\gamma)}{|x|^\gamma} \int_{\frac{|x|}{2}<|t|<2 |x|}  \frac{dt }{|x-t|^{\beta}}.  
	 $$
	 Clearly, 
	 \begin{equation*}
	 \begin{aligned}
	I V & \le  
	 \frac{c(\gamma)}{|x|^{\gamma}} \int_{|x-t|<3 |x|}  \frac{dt }{|x-t|^{\beta}}  \\ & =  \frac{\omega_n c(\gamma)}{|x|^{\gamma}} \int_{0}^{3|x|} r^{n-1-\beta} dr \\ & = 
	 \frac{\omega_n c(\gamma)}{|x|^{\beta + \gamma-n}} = c(\beta, \gamma, n) 
	 |x|^{\epsilon},
	 \end{aligned}
	\end{equation*}
	so that $I V\le c(\beta, n) \, III$. 
	
	Finally, the  integral $V$ is bounded above and below by 
	\begin{equation*}
	 \begin{aligned}
	& c(\beta) \int_{2 |x|<|t|<R}  \frac{dt }{|t|^{\gamma + \beta}} 
	\\ & = c(\beta) \int_{2 |x|}^{R} r^{n-1-\gamma-\beta} dr\\ & = c(\beta) \frac{R^{\epsilon} - (2 |x|)^{\epsilon}}{\epsilon}. 
	 \end{aligned}
	\end{equation*}
	Combining these estimates we complete the proof of \eqref{mathcal-Kb}. 
	  \end{proof}
	 
	 By Lemma \ref{lemma} with $\beta=(n-2 \alpha)q$, $R=R_k$, $\epsilon=\epsilon_k$, 
	 and $\gamma=\gamma_k = n-\beta - \epsilon_k$, we obtain, for $k=2, 3, \ldots$,  
	 \begin{equation}\label{mathcal-Ka} 
	  \begin{aligned}
	 \phi_{R_k, \gamma_k} (x-x_k) & = \int_{|t+x_k|<R_k}  \frac{dt }{|x-t|^{(n-2 \alpha) q}|t+x_k|^{\gamma}} \\& \le C(\alpha, q, n) \, \left\{ \begin{array}{ll} 
	  \frac{R_k^{\epsilon_k}}{\epsilon_k} & \quad \textrm{ if ~~} \, \,   |x-x_k|< 1,\\ & \\
	 \frac{R_k^{\epsilon_k}-1}{\epsilon_k} & \quad \textrm{ if ~~} \, \,   1\le  |x-x_k|\le  \frac{R_k}{2},\\ & \\ R_k^{\epsilon_k} & \quad \textrm{ if ~~} \, \,  |x-x_k| >
	\frac{R_k}{2}.
	\end{array} \right.
	\end{aligned}
	  \end{equation}
	  In the case $k=1$, we use the estimate  
	  $\phi_{R_1, \gamma_1} (x-x_1)\le C(\alpha, q, n)
	   \frac{R_1^{\epsilon_1}}{\epsilon_1}$ for all $x \in \R^n$.

	 We next estimate 
	 \begin{equation}\label{mathcal-K} 
	 \begin{aligned}
	\K(\sigma) & = \sup_{x\in \R^n} \sum_{k=1}^{+\infty} c_k \,  \phi_{R_k, \gamma_k} (x-x_k) 
	\\
	&  \le 
	 \sup_{x\in \R^n} \sum_{|x-x_k| \le 1} c_k \,  \phi_{R_k, \gamma_k} (x-x_k) \\& + \sup_{x\in \R^n}  \sum_{1< |x-x_k| <\frac{R_k}{2}} c_k \,  \phi_{R_k, \gamma_k} (x-x_k) \\ & +  \sup_{x\in \R^n} \sum_{|x-x_k| \ge \frac{R_k}{2}} c_k \,  \phi_{R_k, \gamma_k} (x-x_k) \\ & \defeq  I + II + III.
	 \end{aligned}
	\end{equation}

	Suppose that $j \le |x| \le j+1$ for some $j=0, 1, \ldots$. We first estimate $I$. Since $|x-x_k|\le 1$, 
	and $|x_k|=k$, 
	it follows that 
	$$k=|x_k| \le 1 +|x| \le 1+ |x-x_k| + |x_k|= k+2.$$ Consequently, 
	$j-1\le k \le j+2$ if $j \ge 2$, and $1\le k \le 3$ if $j=0, 1, 2$. Hence, the corresponding sum contains no more 
	than four terms, and therefore 
	\begin{equation*}
	 \begin{aligned}
	I & \defeq \sup_{x\in \R^n} \sum_{|x-x_k| \le 1} c_k \,  \phi_{R_k, \gamma_k} (x-x_k) \\ &  \le C(\alpha, q, n) \, \sup_{j\ge 0} \, \, \sum_{\max(j-1, 1)  \le k \le \max (j+2, 3)} \frac{c_k \, R_k^{\epsilon_k}}{\epsilon_k}\\ & \le C(\alpha, q, n),  
	 \end{aligned}
	\end{equation*}
	since by \eqref{r-k} and \eqref{sums},  
	$$\sup_{k\ge 1} \, R_k^{\epsilon_k}< +\infty, \quad {\rm and} \quad \sup_{k\ge 1} \, \frac{c_k}{\epsilon_k} < +\infty.$$

	To estimate $II$, notice that $0< \epsilon_k \log R_k \le C$, and consequently 
	$$
	\frac{R_k^{\epsilon_k}-1}{\epsilon_k}\le C \log R_k. 
	$$
	Hence, 
	by \eqref{mathcal-Ka} and \eqref{choice}, 
	\begin{equation*}
	 \begin{aligned}
	II & \defeq \sup_{x\in \R^n} \sum_{1< |x-x_k| <\frac{R_k}{2}} c_k \,  \phi_{R_k, \gamma_k} (x-x_k) \\ & \le 
	C(\alpha, q, n)  \sup_{x\in \R^n} \sum_{1< |x-x_k| <\frac{R_k}{2}} \frac{c_k \, (R_k^{\epsilon_k}-1)}{\epsilon_k} \\ & \le C(\alpha, q, n) \sum_{k=1}^{+\infty} \, c_k \, \log R_k < +\infty.   
	 \end{aligned}
	\end{equation*}

	Finally, we estimate $III$ using \eqref{mathcal-Ka} and \eqref{choice}. Since $\sup_k  R_k^{\epsilon_k} < +\infty$, we deduce 
	\begin{equation*}
	 \begin{aligned}
	III & \defeq\sup_{x\in \R^n}  \sum_{|x-x_k| \ge \frac{R_k}{2}} c_k \,  \phi_{R_k, \gamma_k} (x-x_k)\\
	& \le C(\alpha, q, n)   \sup_{x\in \R^n}  \sum_{|x-x_k| \ge \frac{R_k}{2}} c_k \, R_k^{\epsilon_k} 
	\\
	&    \le C(\alpha, q, n)  
	  \sum_{k=1}^{+\infty} c_k  \le C(\alpha, q, n).  
	\end{aligned}
	\end{equation*}
	 This proves \eqref{B}.

	 It remains to verify \eqref{C} for $\sigma=\sum_{k=1}^{+\infty} c_k   \sigma_k$ and $\nu = \sum_{j=1}^{+\infty} a_j \delta_{x_j}$ defined above. We estimate  
	 \begin{equation*}
	 \begin{aligned}
	 \Vert \I_{2 \alpha} \nu\Vert^q_{L^q(\sigma)} & =
	  \sum_{k=1}^{+\infty} c_k  \int_{\R^n} \Big ( \sum_{j=1}^{+\infty} \frac{a_j}{|x+x_j|^{n-2 \alpha} }
	  \Big)^q d \sigma_k \\ & \ge  \sum_{k=1}^{+\infty} c_k  \int_{\R^n} \frac{a_k^q}{|x+x_k|^{q(n-2 \alpha)} }
	 d \sigma_k 
	  \\ & = \sum_{k=1}^{+\infty} c_k \, a_k^q \int_{|x+x_k|<R_k} 
	  \frac{dx}{|x+x_k|^{(n-2 \alpha)q+\gamma_k}}. 
	 \end{aligned}
	\end{equation*} 
	Since 
	\begin{equation*}
	 \begin{aligned}
	\int_{|x+x_k|<R_k} \frac{dx}{|x+x_k|^{(n-2 \alpha)q+\gamma_k}}& = \int_{|x|<R_k} \frac{dx}{|x|^{(n-2 \alpha)q+\gamma_k}} \\ &= \omega_n \int_{0}^{R_k} r^{-1+ \epsilon_k} dr=\omega_n \frac{R_k^{\epsilon_k}}{\epsilon_k}, 
	\end{aligned}
	\end{equation*}
	and $R_k^{\epsilon_k}\ge 1$, it follows  by \eqref{choice} that 
	\begin{equation*}
	 \begin{aligned}
	  \Vert \I_{2 \alpha} \nu\Vert^q_{L^q(\sigma)}  & \ge  \omega_n  \sum_{k=1}^{+\infty} \frac{c_k \, a_k^q}{\epsilon_k}\\ & =  \omega_n  \sum_{k=1}^{+\infty}\frac{1}{k \log(k+1)}=+\infty.
	\end{aligned}
	\end{equation*}

	%%%%%%%%%%%%%%%%%%%%%%%%%%%%%%%%%%%%%%%%%%%%%%%%
%%%%%%%%%%%% Bibliography %%%%%%%%%%%%%%%%%%%%%%
%%%%%%%%%%%%%%%%%%%%%%%%%%%%%%%%%%%%%%%%%%%%%%%%
%\bibliographystyle{siam}
%\bibliography{MyBibliography.bib}

\begin{thebibliography}{25}


  
  \bibitem{AH}  {\sc D.~R. Adams and L.~I. Hedberg,} {\em Function Spaces and Potential Theory},
  Grundlehren der math. Wissenschaften {\bf 314}, Berlin--Heidelberg--New York, Springer, 1996.
  
\bibitem{An} {\sc A. Ancona,} {\em Some results and examples about
the behaviour of harmonic functions and Green's functions with
respect to second order elliptic operators,} Nagoya Math. J. {\bf
165} (2002), 123--158.

  
 % \bibitem{BG}{\sc C. Berge and A. Ghouila-Houri}, {\em Programming, Games and Transportation Networks}, Methuen, London, 1965.

\bibitem{Brelot}
{\sc M.~Brelot}, {\em Lectures on Potential Theory},   Lectures on Math. {\bf 19}, Tata
  Institute, Bombay, 1960.

\bibitem{COV06}
{\sc C.~Cascante, J.~M. Ortega, and I.~E. Verbitsky}, {\em On {$L\sp p$}-{$L\sp
  q$} trace inequalities}, J. London Math. Soc.  {\bf 74} (2006), 497--511.

\bibitem{CV1}
{\sc C.~T. Dat and I.~E. Verbitsky}, {\em Finite energy solutions of
  quasilinear elliptic equations with sub-natural growth terms}, Calc. Var.
  PDE {\bf 52} (2015), 529--546.
  
  \bibitem{CV2} {\sc Dat T. Cao and I. E. Verbitsky}, {\em Nonlinear elliptic equations and intrinsic potentials of Wolff type,} 
J. Funct. Analysis {\bf 272} (2017), 112--165. 


\bibitem{CV3} {\sc Dat T. Cao and I. E. Verbitsky}, {\em Pointwise estimates of Brezis--Kamin type for solutions of sublinear elliptic equations,}  Nonlin. Analysis, Ser. A: Theory, Methods \& Appl. {\bf 146} (2016), 1--19. 

\bibitem{FNV}
{\sc M.~Frazier, F.~Nazarov, and I.~Verbitsky}, {\em Global estimates for
  kernels of Neumann series and Green's functions},  J. London Math. Soc. {\bf 90}   (2014),  903--918.

\bibitem{FV}
{\sc M.~Frazier and I.~Verbitsky}, {\em Positive solutions to Schr\"odinger's equation and the exponential integrability of the balayage},  Ann. Inst. Fourier (Grenoble) (published online), arXiv:1509.09005.

\bibitem{Frostman}
{\sc O.~Frostman}, {\em Potentiel de masses \`a somme alg\'ebrique nulle},
  Kungl. Fysiogr. S\"allskapets i Lund F\"orhandlingar [Proc. Roy.
  Physiog. Soc. Lund] {\bf 20} (1950), 1--21.

\bibitem{F}
{\sc B.~Fuglede}, {\em On the theory of potentials in locally compact spaces},
  Acta Math. {\bf 103} (1960), 139--215.

\bibitem{G}
{\sc E.~Gagliardo}, {\em On integral trasformations with positive kernel},
  Proc. Amer. Math. Soc. {\bf 16} (1965), 429--434.

 \bibitem{GV1}  {\sc A. Grigor'yan and I. E. Verbitsky,} {\em Pointwise estimates of solutions to
semilinear elliptic equations and inequalities}, J. d'Analyse Math. (to appear) arXiv:1511.03188.  



 \bibitem{GV2}  {\sc A. Grigor'yan and I. E. Verbitsky,} {\em  Pointwise estimates of solutions to nonlinear equations for nonlocal operators}, arXiv:1707.09596. 

\bibitem{HN}
{\sc W.~Hansen and I.~Netuka}, {\em On the Picard principle for ${\Delta} +
  \mu$}, Math. Z. {\bf 270} (2012) 783--807.

  
  \bibitem{MV} {\sc M. Marcus and L. V\'eron,} {\em Nonlinear Second Order Elliptic Equations Involving Measures}, Walter de Gruyter, Berlin--Boston, 2014. 

\bibitem{M}
{\sc B.~Maurey}, {\em Th\'eor\`emes de factorisation pour les op\'erateurs
  lin\'eaires \`a valeurs dans les espaces {$L\sp{p}$}}, Ast{\'e}risque {\bf 11} (1974) 
  Soc. Math. France, Paris.
  
  \bibitem{Maz} {\sc V. Maz'ya}, {\em Sobolev Spaces, with Applications to Elliptic Partial Differential Equations}, 2nd, Augmented Edition. Grundlehren der math. Wissenschaften \textbf{342}, Springer, Berlin, 2011.

\bibitem{QV1}
{\sc S.~Quinn and I.~E. Verbitsky}, {\em Weighted norm inequalities of
  $(1,q)$-type for integral and fractional maximal operators},  Harmonic Analysis, Partial Differential Equations
and Applications, in Honor of Richard L. Wheeden,  eds. S. Chanillo et al., Ser. Appl. Numer. Harmonic Analysis, Birkh\"auser,  2017,  217--238.  

  
  \bibitem{QV2}
{\sc S.~Quinn and I.~E. Verbitsky}, {\em A sublinear version of Schur's lemma and elliptic PDE}, 
Analysis \& PDE (to appear), arXiv:1702.02682.  

\bibitem{SW}
{\sc E.~Stein and G.~Weiss}, {\em Introduction to Fourier Analysis on Euclidean
  Spaces}, Princeton Univ. Press, Princeton, N.J., 1971.

\bibitem{Szeptycki}
{\sc P.~Szeptycki}, {\em Notes on Integral Transformations}, Dissert. 
  Math. (Rozprawy Mat.), 231 (1984), p.~48.

 
\end{thebibliography}

\end{document}